\documentclass[12pt]{article}
 
\usepackage[margin=1in]{geometry} 
\usepackage{amsmath,amsthm,amssymb}
\usepackage{mathtools}
\usepackage{cite}
\usepackage{graphicx}
\usepackage{cleveref}
\usepackage{url}

\newcommand{\F}{\mathbb{F}}

\newtheorem{theorem}{Theorem}
\newtheorem{definition}{Definition}
\newtheorem{corollary}{Corollary}[theorem]
\newtheorem{lemma}[theorem]{Lemma}
\newtheorem{proposition}{Proposition}
\newtheorem{question}{Question}

\begin{document}
 
%
%
 
\title{Permutations minimizing the number of collinear triples} 
\author{Joshua Cooper and Jack Hyatt}
\maketitle

\begin{abstract}
    We characterize the permutations of $\mathbb{F}_q$ whose graph minimizes the number of collinear triples and describe the lexicographically-least one, affirming a conjecture of Cooper-Solymosi.  This question is closely connected to Dudeney's No-3-in-a-Line problem, the Heilbronn triangle problem, and the structure of finite plane Kakeya sets.  We discuss a connection with complete sets of mutually orthogonal latin squares and state a few open problems primarily about general finite affine planes.
\end{abstract}

\section{Introduction}

The classic open ``no $3$ in a line'' problem of Dudeney (1917) asks: what is the maximum number of points that can be placed on an $n \times n$ grid so that there are no three collinear points?  The answer is certainly no more than $2n$, since any $2n+1$ points contain at least $3$ points on at least one of the $n$ vertical lines.  Even though it is conjectured that, in fact, no more than $\pi/\sqrt{3} n \approx 1.814 n$ points can be placed without any collinear triples (\cite{guy}, corrected by \cite{pegg}), not even $2n-1$ is known.  This problem is closely connected to the classical Heilbronn triangle problem, and Erd\H{o}s gave a lower bound of $n-o(n)$ (cf.~\cite{roth}), later improved to $3n/2-o(n)$ in \cite{hall}.  It is also closely related to the finite field Kakeya problem (cf.~\cite{blokhuis}), as employed in \cite{li}: What is the fewest number of points in a subset of the affine plane containing a line with every slope?  It is straightforward to see that any collinear-triple-free set with $2n$ points can be partitioned into the graphs of two permutations.  Thus, \cite{cooper} studied collinearities in permutations -- discussed below, after a few definitions.

Throughout the paper, we use notation that $\F_q$ is the finite field with order a prime power, $q$; at times we suppress the subscript and simply write $\F$. Given a function $f : \F \rightarrow \F$, the graph of $f$ is the set $\{(x,f(x)) : x \in \F\}$.  The affine geometry over $\F$ of dimension $d$, denoted $AG(q,d)$, consists of the points $\F^d$ and all lines of the form $\{\mathbf{v}+t\mathbf{w} : t \in \F\}$ where $\mathbf{v},\mathbf{w} \in \F^d$ and $\mathbf{w}$ is not the all-zero vector $\hat{0}$.  The projective geometry $PG(q,d)$ over $\F$ of dimension $d$ has point set $(\F^{d+1} \setminus \{\hat{0}\}) /\sim$, where $\sim$ is the equivalence relation defined by: $\mathbf{v} \sim \mathbf{w}$ iff there exists $\lambda \in \F^\times$ so that $\mathbf{w} = \lambda \mathbf{v}$; the lines $\ell$ of $PG(q,d)$ are obtained from each $2$-dimensional subspace $P \subseteq AG(q,d)$ by taking $\ell = (P \setminus \{\hat{0}\})/\sim$.  The $\sim$-equivalence class of the points $(x_1,\ldots,x_{d+1}) \in AG(q,d)$ is denoted $[x_1 : \ldots : x_{d+1}]$.  We take $d=2$ throughout the sequel -- i.e., affine and projective planes -- and note that the graph of a function can be considered a subset of $AG(q,2)$.

\begin{definition}
    A collinear triple in a permutation, $\sigma$ of $\F$, is a set of three distinct points $(x_1,\sigma(x_1)),(x_2,\sigma(x_2)),(x_3,\sigma(x_3))$ of the graph of $\sigma$ for which there exists a line $\ell \in AG(q,2)$ containing each $(x_i,\sigma(x_i))$, $i=1,2,3$.  Equivalently,
    \[
    x_1(\sigma(x_2)-\sigma(x_3))+x_2(\sigma(x_3)-\sigma(x_1)) + x_3(\sigma(x_1)-\sigma(x_2))=0.
    \]
\end{definition}

In \cite{cooper}, it was conjectured that the best lower bound for the number of collinear triples in a permutation over $\F$, where $q>2$ is prime, is
\begin{equation}\label{triples}
    \Psi(q) \geq \frac{q-1}{2}.
\end{equation}
In \cite{li}, Li confirmed this conjecture by invoking the main result of \cite{blokhuis}. It was also conjectured in \cite{cooper} that the permutation 
\[
g(x)=\begin{cases}
		1, & \text{if $x=1$}\\
            x/(x-1), & \text{otherwise}
		 \end{cases}
\]
is the lexicographic-least permutation that satisfies the lower bound of $(q-1)/2$ collinear triples. In the next section, we confirm that conjecture as well, specifically in \Cref{thm:lexleast}.  In the last section, we discuss a connection with complete sets of mutually orthogonal latin squares, present some intriguing computations on affine planes arising from non-Desarguesian projective planes, and state a few open problems.

\section{Main Results}

We assume $q$ is an odd prime power throughout.  Let $\sigma$ be a permutation over $\F_q$, and $\Gamma$ be the graph of $\sigma$. Assume $\Gamma$ minimizes the number of collinear triples, i.e. the number of collinear triples in $\Gamma$ is $(q-1)/2$. Map $\Gamma$ to $PG(2,q)$ by $\rho : (x,y) \mapsto [x : y : 1]$, and construct $\Omega \coloneqq \rho(\Gamma) \cup \{[1:0:0],[0:1:0]\}$. The plane $z=0$ is the {\em line at infinity}, denoted $\ell_\infty$.

\begin{definition}
    For $q$ a prime power, let $\Omega$ be a set of $q+2$ points in $PG(2,q)$. A point $P \in \Omega$ is an \emph{internal nucleus} of $\Omega$ if every line through $P$ meets $\Omega$ in exactly one other point.
\end{definition}

Since $\sigma$ is a permutation with $q$ elements and $\{[1:0:0],[0:1:0]\} \notin \rho(\Gamma)$, we have that $|\Omega| = q+2$.  Note that $[1:0:0]$ is an internal nucleus (as well as $[0:1:0]$, but we only need one), as the projective line between $[1:0:0]$ and $[x_0:y_0:z_0]$ is $z_0y-y_0z=0$. For points arising from the original permutation, the line is $y=y_0z$, which is unique since each element in a permutation has a unique $y_0$. For the other added point, $[0:1:0]$, the line is $\ell_\infty$.  

In the following result, we consider the minimum number of lines intersecting $\Omega$, denoted $k^*(\Omega)$.  

\begin{proposition} \label{prop:kstar}
\begin{equation} \label{eqkstar}
k^*(\Omega) = \frac{(q+1)(q+2)}{2} + \frac{q-1}{2}.
\end{equation}
\end{proposition}

\begin{proof}

Using the inclusion-exclusion principle, we first add the number of lines through each point, then subtract off the number of lines through pairs of points, then add back number of lines through three points. By Bonferroni's Inequality, the result is an upper bound for $k^*(\Omega)$ (for any $\Omega$ consisting of $q+2$ points, including an internal nucleus). So we end up with
\[
k^*(\Omega) \leq (q+1)(q+2) - \binom{q+2}{2} + \frac{q-1}{2} = \frac{(q+1)(q+2)}{2} + \frac{q-1}{2}.
\]
On the other hand, by \cite{blokhuis}, the above quantity is also a lower bound, so we have equality.
\end{proof}

A {\em conic} in $PG(q,2)$ is the solution set to a homogeneous polynomial $p(x,y,z)$ of degree $2$.  If $p \not \equiv 0$, then we say that $p$ is ``nondenegerate'', and the conic is said to be ``irreducible'' if $p$ is irreducible over $\F_q$.  Blokhuis and Mazzocca showed in \cite{blokhuis} that any $\Omega$ with $|\Omega|=q+2$, including an internal nucleus, which satisfies (\ref{eqkstar}), is a nondegenerate irreducible conic with an external point.  Therefore, we have the following.

\begin{corollary}
$\Omega$ is a nondegenerate irreducible conic with an external point.    
\end{corollary}

Let us then rewrite $\Omega = \{[x:y:z] : ax^2 + by^2 + cxy + dxz +eyz + fz^2 = 0,\ (x,y,z)\neq(0,0,0)\} \cup \{ [x':y':z'] \}$ where $[x':y':z']$ is the external point, $(x',y',z') \neq (0,0,0)$. Note that the $q+1$ points that are not the external point will define the conic.  Thus, $q+1$ points determine the 6 coefficients of the conic's polynomial $p$.  Only $5$ points are needed, however, since the homogeneity of $p$ entails only $5$ degrees of freedom, and the irreducibility of $p$ implies that it is uniquely determined by these $5$ points.  (See, for example, \cite{coxeter}.)  As long as $q \geq 5$, any two nondegenerate conics obtained from $q+1$ points of $\Omega$ will share at least 5 points, as the only two points to differ will be the two excluded points. This means they are the same conic. So there is exactly one nondegenerate conic obtained from the point set $\Omega$, and the external point is the unique point not on this conic.  Shortly, we determine which of the $q+2$ elements of $\Omega$ is the external point.

Now we remove $[1:0:0]$ and $[0:1:0]$ and map back to the affine plane by $\rho^{-1}:[x:y:1] \mapsto (x,y)$. As the only thing to happen in this process was adding and removing the points $[1:0:0]$ and $[0:1:0]$, the permutation itself remains unchanged.  We obtain two possibilities for the image of $\Gamma$ under $\rho$, restriction to $z \neq 0$, and then $\rho^{-1}$.

If the external point was either $[1:0:0]$ or $[0:1:0]$, then
\[
\Gamma = \{(x,y): ax^2 + by^2 + cxy + dx +ey + f = 0\}.
\]
If the external point was neither $[1:0:0]$ nor $[0:1:0]$, then
\[
\Gamma = \{(x,y): ax^2 + by^2 + cxy + dx +ey + f = 0\} \cup \{(x',y')\}.
\]
where $[x':y':1]$ is the external point. We shall denote $\Gamma'$ as just the conic, removing the external point if we are in the latter of the above two cases.

\begin{lemma}\label{lem:ab0}
    For an odd prime power $q \geq 5$, $\Gamma'$ is of the form $\{(x,y): cxy + dx +ey + f = 0\}$, i.e., $a=b=0$ in $p$.
\end{lemma}

\begin{proof}

If $b\neq 0$, then the discriminant of the quadratic obtained by solving for $y$ is $(cx+e)^2-4b(ax^2+dx+f)$, which must be $0$, or else there are $0$ or $2$ solutions to $ax^2+by^2+cxy+dx+ey+f=0$ for some $x$, contradicting that $\Gamma'$ is the graph of a function.

Expanding gives
\[
(4ab-c^2)x^2+(4bd-2ce)x+(4bf-e^2)=0.
\]
This quadratic would need to hold for every $x\in\F_q$, except the $x$ corresponding to the external point ignored in $\Gamma'$. That cannot be the case however when $q\geq4$, since quadratics only have at most 2 unique solutions, unless all coefficients are $0$, i.e., $4ab=c^2$, $2bd=ce$, and $4bf=e^2$.   Multiplying the conic by $4b$, the conic becomes
\begin{align*}
4abx^2+4b^2y^2+4bf+4bcxy+4bdx+4bey &= c^2 x^2+4b^2 y^2+e^2+4bcxy+2cex+4bey \\
&= (cx + 2by + e)^2,
\end{align*}
which is reducible, a contradiction.  Thus, $b=0$.  By an identical argument swapping $x$ and $y$, we conclude that $a=0$. 

\end{proof}

\begin{corollary}\label{cor:frac}
For an odd prime power $q$, $\Gamma'$ is the graph of a fractional linear transformation.
\end{corollary}

\begin{proof}
For $q \geq 5$, we have that $\Gamma' = \{(x,y): cxy + dx +ey + f = 0\}$ from \Cref{lem:ab0}. Simple manipulation gives
\[
\Gamma' = \{(x,y) : y = (dx+f)/(cx+e)\}.
\]
It is easy to check that the result also holds for $q=3$.
\end{proof}

\begin{corollary}\label{cor:hasext}
The external point in $\Gamma$ is not on the line of infinity in $PG(2,q)$, i.e. $\Gamma = \{(x,y): cxy + dx +ey + f = 0\} \cup \{(x',y')\}$.
\end{corollary}

\begin{proof}
By \Cref{lem:ab0}, we have that $\Omega = \{[x:y:z] : cxy + dxz +eyz + fz^2 = 0,\ (x,y,z)\neq(0,0,0)\} \cup \{ [x':y':z'] \}$, where $[x':y':z']$ is the external point not on the conic. Clearly, $[0:1:0]$ and $[1:0:0]$ both are points that satisfy the equation $cxy + dxz +eyz + fz^2 = 0$, which therefore means $[0:1:0]$ and $[1:0:0]$ were on the conic, i.e., neither is the external point.
\end{proof}

\begin{definition}
    For $q$ a prime power, an \em{extended linear fractional transformation permutation with parameters $\alpha,\beta,\gamma$} with $\alpha \gamma \neq \beta$, or an \em{ELFTP} for short, is a function of $\mathbb{F}_q$ of the form
\begin{equation} \label{eq1}
\sigma(x)=  \begin{cases}
                (\alpha x+\beta)/(x+\gamma), & \text{if $x \neq -\gamma$}\\
		      \alpha, & \text{if $x=-\gamma$}
            \end{cases}
\end{equation}
\end{definition}
It is easy to see that ELFTPs are indeed permutations.

\begin{theorem}\label{thm:fractional}
Let $q>2$ be an odd prime power and $\sigma$ a permutation that admits precisely $(q-1)/2$ collinear triples. Then $\sigma$ is an ELFTP.
\end{theorem}

\begin{proof}
By \Cref{cor:frac} and \Cref{cor:hasext}, we have that
\[
\Gamma = \{(x,\sigma(x)\} = \{(x,y) : y = (dx+f)/(cx+e)\}\cup \{(x',y')\}.
\]
For simplicity, we can divide through by $c$, and then just relabel the coefficients to look like \Cref{eq1}. That is,
\[
\frac{(d/c)x + (f/c)}{x + (e/c)} = \frac{\alpha x+\beta}{x+\gamma}.
\]

Then we are left with
\[
\Gamma = \left\{ \left(x,\frac{\alpha x+\beta}{x+\gamma}\right): x \in \F_q-\{x'\} \right\} \cup \{(x',y')\}
\]
where $\alpha \gamma \neq \beta$. We claim that $(x',y') = (-\gamma,\alpha)$. The domain of $(\alpha x+\beta)(x+\gamma)$ does not include $-\gamma$, so $x'=-\gamma$. We then can manipulate the fraction to show
\begin{align*}
    \frac{\alpha x+\beta}{x+\gamma} &= \frac{\alpha x + \alpha \gamma + \beta - \alpha \gamma}{x+\gamma} = \alpha + \frac{\beta - \alpha \gamma}{x+\gamma}.
\end{align*}
Using the fact that $\alpha\gamma \neq \beta$, we have that $\alpha + (\beta-\alpha\gamma)/(x+\gamma)$ is never equal to $\alpha$. Since $\Gamma$ is the graph of a permutation, there must be a point $(x,y)\in\Gamma$ where $y=\alpha$, so our external point is $(-\gamma,\alpha)$. 

Therefore $\sigma$ is an ELFTP.
\end{proof}

\begin{theorem}\label{thm:lexleast}
For an odd prime power $q$, the lexicographic-least permutation minimizing the number of collinear triples is
\[
g(x)=\begin{cases}
		1, & \text{if $x=1$}\\
            x/(x-1), & \text{otherwise}
		 \end{cases}
\]
\end{theorem}

\begin{proof}
    Let $\sigma(x)$ be a permutation that admits $(q-1)/2$ collinear triples; then, by \Cref{thm:fractional}, it is an ELFTP. We argue that there is exactly one such permutation satisfying $\sigma(0)=0$, $\sigma(1)=1$, and $\sigma(2)=2$, from which it follows that it must be the lexicographic-least. There are 4 cases: $\gamma\notin\{0,-1,-2\}$, $\gamma =0$, $\gamma =-1$, and $\gamma =-2$.

    \textbf{Case 1:} $\gamma \notin\{0,-1,-2\}$\\
    Then $\alpha x+\beta=\sigma(x)(x+\gamma )$ for $x \in \{0,1,2\}$. That $\sigma(0)=0$ implies $\beta =0$. Since $\sigma(1)=1$ and $\sigma(2)=2$, it follows that $\alpha =1+\gamma $ and $2\alpha=4+2\gamma$, a contradiction.

    \textbf{Case 2:} $\gamma =0$\\
    Then $\alpha x+\beta = x\sigma(x)$ for $x \in \{1,2\}$. Since $\sigma(1)=1$ and $\sigma(2)=2$, it follows that $\alpha +\beta=1$ and $2\alpha+\beta=4$. This implies $\beta =-2$ and $\alpha =3$. But since $\sigma(-\gamma)=\alpha$, we may conclude that $\alpha =\sigma(0)=0$, a contradiction.

    \textbf{Case 3:} $\gamma =-2$\\
    Then $\alpha x+\beta = \sigma(x)(x-2)$ for $x \in \{0,1\}$. Since $\sigma(0)=0$ and $\sigma(1)=1$, it follows that $\beta =0$ and $\alpha +\beta=-1$. This implies $\alpha =-1$. But since $\sigma(-\gamma)=\alpha$, we may conclude that $\alpha =\sigma(2)=2$, a contradiction.

    \textbf{Case 4:} $\gamma =-1$\\
    Then $\alpha x+\beta = \sigma(x)(x-1)$ for $x \in \{0,2\}$. Since $\sigma(0)=0$ and $\sigma(2)=2$, it follows that $\beta =0$ and $2\alpha+\beta=2$. This implies $\alpha =1$. Since $\sigma(-c)=a$, we may conclude that $\sigma(1)=1$. So $\alpha =1$, $\beta =0$, $\gamma =-1$ yields the unique permutation of form \Cref{eq1} with $(\sigma(0),\sigma(1),\sigma(2))=(0,1,2)$.

    Thus, there is exactly one ELFTP $\sigma$ that satisfies $\sigma(0)=0$, $\sigma(1)=1$, and $\sigma(2)=2$, and it is the desired lexicographically-least permutation.
\end{proof}

\section{Conclusion}

It is worth noting that, although \cite{cooper} concerned itself primarily with permutations of $\mathbb{Z}_p$ for $p$ prime, some of the results therein apply to more general geometries than just  finite $2$-dimensional vector spaces of odd prime order.  In particular, we generalize here the short proof that every graph of a permutation of $\mathbb{Z}_p$ admits three collinear points, so that it applies to any finite affine plane.  Given an affine plane $(\mathcal{P},\mathcal{L})$ and two parallel classes $H$ and $V$ of its lines, define a ``generalized permutation'' to be a set $S \subset P$ which is a transversal to both $H$ and $V$, i.e., so that, for each $\ell \in H \cup V$, there exists exactly one $p \in S$ so that $p \in \ell$.

\begin{theorem} \label{thm:oldthm}
    Every generalized permutation in an affine plane of odd order $q$ admits a collinear triple.
\end{theorem}
\begin{proof}
    Let $\{L_1,\ldots,L_{q-1},H,V\}$ be a partition of $\mathcal{P}$ into $q+1$ parallel classes of $q$ lines each.  Consider a generalized permutation $S$ and its set $\binom{S}{2}$ of unordered pairs. Since no such pairs are contained in an element of $H$ or $V$, for all $\{p_1,p_2\} \in \binom{S}{2}$, there exists exactly one $j \in [q-1]$ so that there exists an $\ell \in L_j$ with $\{p_1,p_2\} \subset L_j$.  By the pigeonhole principle, some $L_j$ contains at least $\lceil \binom{q}{2}/(q-1) \rceil= \lceil q/2 \rceil = (q+1)/2$ pairs.  These pairs cannot all be contained in distinct (and therefore disjoint) elements of $L_j$, or else we would have $\geq 2 \cdot (q+1)/2 = q+1$ points in $S$, a contradiction.  Thus, two pairs are contained in some $\ell \in L_j$, which implies that $\ell \cap S$ contains at least three points.
\end{proof}

For convenience, we denote $\{1,\ldots,n\}$ by $[n]$.  An interesting restatement of the above result takes advantage of the well-known connection between affine planes and complete sets of mutually orthogonal latin squares (MOLs): {\em For every complete set $M_1,\ldots,M_{q-1}$ of MOLs of odd order $q$, there exists a $j \in [q-1]$ and $k \in [q]$ so that the entry $k$ repeats at least $3$ times on the diagonal of $M_j$.}  To see this, consider the MOLs arising from an affine plane with distinguished parallel classes $H = \{h_1,\ldots,h_q\}$ and $V = \{v_1,\ldots,v_q\}$.  The points of $\mathbb{A}$ can be coordinatized as $[q] \times [q]$ by identifying $p$ with $(i,j)$ if $h_i \cap v_j = p$.  A generalized permutation $S$ provides a bijection between $H$ and $V$, which we may assume sends $h_i$ to $v_i$ for each $i$, since the labels on $H$ and $V$ are arbitrary.  That $\{p_1,p_2,p_3\} \subseteq S$ is collinear is the fact that, for some $j \in [q-1]$ and $k \in [q]$, the entries of $M_j$ where $k$ occurs contains $\{p_1,p_2,p_3\}$.  Since the labeling of $H$ and $V$ was chosen so that, for each $i \in [3]$, $p_i = (t_i,t_i)$ for some $t_i$, this means that the $(t_i,t_i)$ entry of $L_j$ is $k$, i.e., $L_j$ has at least three $k$ entries on its diagonal.

Thus, it is natural to ask the following generalization of the question resolved by Li in \cite{li} and discussed in the introduction above.  For an affine plane $\mathbb{A}$, let $\Psi(\mathbb{A})$ denote the fewest number of collinear triples in any generalized permutation of $\mathbb{A}$.  It is not hard to see that, if $\mathbb{A} = AG(\F_q,2)$ with $q = 2^k$, the minimum number of collinear triples in a permutation is $0$.  Indeed, let $\sigma : \F_q \rightarrow \F_q$ be defined by $\sigma(x)=1/x$ if $x \neq 0$, and $\sigma(0)=0$.  For $x,y,z \neq 0$ distinct elements of $\F_q$, collinearity amounts to the equation
$$
\frac{1/x-1/y}{x-y} = \frac{1/x-1/z}{x-z}
$$
which simplifies to $(z-x)(y-x)(x-z)=0$, a contradiction.  If $z=0$, the equation becomes
$$
\frac{1/x-1/y}{x-y} = \frac{1/x}{x}
$$
which simplifies to $y=-x$, i.e., $y=x$ because $\F_q$ has characteristic $2$, a contradiction.  This leaves open the question for other planes than those arising from Desarguesian geometries, so we ask:

\begin{question}
    What is $\Psi(\mathbb{A})$ for affine planes $\mathbb{A}$ other than $AG(\F,2)$ for some finite field $\F$, i.e., those that arise from non-Desarguesian projective planes?
\end{question}

In \cite{cooper}, the authors show in Corollary 2.3 that $\Psi(\mathbb{A}) \geq (q-1)/4$ for $\mathbb{A} = AG(q,2)$ when $q$ is odd.  It is worth noting here that the proof applies to all finite affine planes, {\em mutatis mutandis}, replacing each mention of the collection of lines with a given slope by a parallel class.  However, we do not know if Li's (cf.~\cite{li}) stronger lower bound $\Psi(\mathbb{A}) \geq (q-1)/2$ holds, nor do we have any nontrivial upper bounds.  In fact, we are able to compute this value at least for a few small affine planes not arising from vector spaces.  There are three non-Desarguesian projective planes of order $9$: the ``Hall plane'', its dual, and the ``Hughes plane''.  According to \cite{hurkens} (citing \cite{kamber}), there are seven non-isomorphic affine planes of order $9$, namely, $AG(9,2)$ and six obtained by, for each of the aforementioned projective planes, deleting one line in two different ways. Using data provided by \cite{moore} via SageMath \cite{sage}, we tested the above question for all seven affine planes, obtaining $\Psi(\mathbb{A})=4$ in every case but one: interestingly, the geometry obtained by deleting the unique translation line of the Hall plane has $\Psi(\mathbb{A})=5$.\\

We also reiterate here a conjecture from \cite{cooper} which extends the above question in a different direction.

\begin{question}
    Suppose $\sigma$ is a permutation of $\mathbb{Z}_n$, for $n$ composite.  What is the fewest number $\Psi(n)$ of collinear triples in the graph of $\sigma$?  For which $n$ is it positive?
\end{question}

The quantity $\Psi(n)$ is known for all $n \leq 17$ (\cite[A379299]{oeis}).  Of course, one may also substitute any finite subset of any ring for $\mathbb{Z}_n$ in the above question. \\

An consequence of the argument for Proposition \ref{prop:kstar} is that, if a permutation achieves the minimum number $(q-1)/2$ of collinear triples in $AG(q,2)$, then it has no collinear quadruples, since otherwise Bonferroni's Inequality would be strict and violate the known lower bound on $k^*(\Omega)$.  We therefore ask:

\begin{question}
    Is it true that every generalized permutation in a finite affine plane which minimizes the number of collinear triples admits no collinear quadruples?
\end{question}

\bibliography{biblio}{}
\bibliographystyle{plain}
 
\end{document}